\documentclass[11pt]{amsart}

\usepackage[a4paper,margin=1in]{geometry}
\usepackage{amsmath,amssymb,amsthm,enumerate}
\usepackage{mathrsfs}
\usepackage{hyperref}

\newtheorem{theorem}{Theorem}
\newtheorem{proposition}{Proposition}
\newtheorem{lemma}{Lemma}

\theoremstyle{definition}
\newtheorem{problem}{Problem}

\theoremstyle{remark}
\newtheorem{remark}{Remark}
\theoremstyle{plain}

\title{Shadowing and metric expansivity on Fr\'echet spaces}
\author{Xinxing Wu}
\address{School of Mathematics and Statistics, Guizhou University of Finance and Economics,
Guiyang, Guizhou 550025, China}
\email{wuxinxing5201314@163.com}
\author{Guting Wang}
\address{School of Mathematics and Statistics, Guizhou University of Finance and Economics,
Guiyang, Guizhou 550025, China}
\email{wangguting@mail.gufe.edu.cn}
\date{\today}

\begin{document}

\maketitle

\begin{abstract}
We prove that (1) the weighted composition operator \(Tf(z)=c f(Bz)\)
on the Fr\'echet space \(H(\mathbb{C}^d)\) of entire functions
on \(\mathbb{C}^d\) (\(d\in\mathbb{N}\)) with the compact-open topology
has the shadowing
property whenever \(0<|c|<1\) and \(B\in\operatorname{GL}_d(\mathbb C)\)
has spectral radius less than \(1\); (2) every bilateral weighted forward shift
on \(\mathbb K^{\mathbb Z}\) with nonzero weights has the shadowing property;
and (3) no continuous linear operator on a countably infinite product of
nonzero finite-dimensional normed spaces is
metrically positively expansive for any compatible metric, and no
linear homeomorphism on such a product is metrically expansive for any
compatible metric.
These results answer the \(H(\mathbb C)\)-part of \cite[Problem A]{BCDFP}
and the \(\mathbb K^{\mathbb Z}\)-part of \cite[Problem B]{BCDFP}.
\end{abstract}

\section{Introduction}

Throughout the paper, let \(\mathbb{N}=\{1,2,\ldots\}\) and
\(\mathbb{N}_0=\{0,1,2,\ldots\}.\) The scalar field is
\(\mathbb{K}=\mathbb{R}\) or \(\mathbb C\). For a Fr\'echet space
\(X\), \(L(X)\) denotes the set of continuous linear operators
on \(X\), \(GL(X)\) denotes the set of linear homeomorphisms of \(X\),
equivalently the set of all \(T\in L(X)\) such that \(T^{-1}\in L(X)\),
and \(\operatorname{GL}_d(\mathbb{C})\) denotes the set of invertible
\(d\times d\) matrices with complex entries ($d\in\mathbb{N}$).
A compatible metric is a metric that induces the given
topology.

The shadowing property and expansivity are central notions in topological
dynamics \cite{AokiHiraide,Pilyugin,Walters}.  For a continuous linear
operator, shadowing requires sufficiently accurate pseudotrajectories to be
traced by exact orbits, whereas expansivity requires a compatible metric that
separates distinct orbits by some iterate.  General background on linear
dynamics is given in the monographs of Bayart and Matheron
\cite{BayartMatheron} and Grosse-Erdmann and Peris
\cite{GrosseErdmannPeris}.

In Banach spaces, shadowing and expansivity are closely connected with
generalized hyperbolicity, chain recurrence, and structural stability.
Bernardes and Messaoudi \cite{BernardesMessaoudi} proved shadowing and
structural-stability theorems for generalized hyperbolic operators.  They also
characterized weighted shifts with the shadowing property on
\(c_0(\mathbb Z)\) and \(\ell_p(\mathbb Z)\).  Bernardes and Peris
\cite{BernardesPeris} proved that shadowing and finite shadowing coincide for
Banach-space operators and obtained chain-recurrence criteria for weighted
shifts on Fr\'echet sequence spaces.  Alves et
al.~\cite{AlvesBernardesMessaoudi} related chain recurrence to average
shadowing.  Expansive linear homeomorphisms and early shadowing results for
linear maps were studied in
\cite{Eisenberg,EisenbergHedlund,Hedlund,Ombach,Mazur}.

For Fr\'echet spaces, orbit estimates for shadowing and metric expansivity are
controlled by defining seminorms rather than by a single norm.  Several results
on linear dynamics beyond Banach spaces show how dynamical criteria are
adapted to this locally convex setting.  Bonet and Peris \cite{BonetPeris}
constructed hypercyclic operators on non-normable Fr\'echet spaces.  Bonet et
al.~\cite{BonetFrerickPerisWengenroth} gave criteria for transitivity and
hypercyclicity of operators on locally convex spaces.  Grosse-Erdmann and
Peris \cite{GrosseErdmannPerisWeaklyMixing} obtained weak-mixing criteria for
operators on topological vector spaces, and Peris
\cite{PerisHypercyclicityCriterion} proved a hypercyclicity criterion for
non-metrizable topological vector spaces.  Bonet et
al.~\cite{BonetKalmesPeris} analyzed transitivity, sequential hypercyclicity,
and chaos for shift operators on non-metrizable sequence spaces.  For
operators on locally convex spaces, Bernardes et al.~\cite{BCDFP} developed
generalized hyperbolicity, stability, and expansivity and formulated the
following two existence problems.

\cite[Problem A]{BCDFP} concerns the existence of continuous linear operators with
the shadowing property on Fr\'echet spaces.

\begin{problem}[{\cite[Problem A]{BCDFP}}]\label{problem:A}
Does every Fr\'echet space (or locally convex space) support an
operator with the shadowing property?  How about the Fr\'echet
space \(H(\mathbb{C})\)?
\end{problem}

\cite[Problem B]{BCDFP} concerns metric expansivity of continuous linear operators
on Fr\'echet spaces.

\begin{problem}[{\cite[Problem B]{BCDFP}}]\label{problem:B}
Does every Fr\'echet space (or locally convex space) support an expansive
operator?  How about the Fr\'echet space \(\mathbb{K}^{\mathbb{Z}}\)?
Here expansivity is
meant in the metric sense.
\end{problem}

The present paper treats the two Fr\'echet spaces explicitly named in
Problems~\ref{problem:A} and~\ref{problem:B}: \(H(\mathbb C)\) in
Problem~\ref{problem:A} and \(\mathbb K^{\mathbb Z}\) in
Problem~\ref{problem:B}.  The arguments are expressed in terms of the defining
seminorms of the spaces involved.  On \(H(\mathbb{C}^d)\) these are the
compact-open seminorms, while on countable products they are the
finite-coordinate seminorms
\[
p_F(x)=\max_{i\in F}\|x_i\|,
\]
where \(F\) ranges over nonempty finite subsets of the index set.

The first main result gives an affirmative answer to the \(H(\mathbb C)\)-part
of Problem~\ref{problem:A}.  Section~\ref{Sec:2} proves a summability criterion
for the shadowing property and applies it to weighted composition operators on
\(H(\mathbb{C}^d)\).  Theorem~\ref{thm:HCd-shadowing} shows that, if
\(d\in\mathbb{N}\), \(0<|c|<1\), and \(B\in \operatorname{GL}_d(\mathbb{C})\)
has spectral radius less than \(1\), then
the weighted composition operator
\[
Tf(z)=c f(Bz)
\]
defines an operator in \(GL(H(\mathbb{C}^d))\) with the shadowing property.  In
particular, \(H(\mathbb C)\) supports an operator with the shadowing property.

Section~\ref{Sec:omega-shadowing} also proves a shadowing result on the
product space \(\mathbb{K}^{\mathbb{Z}}\).
Theorem~\ref{thm:omega-weighted-shift-shadowing} states that every bilateral
weighted forward shift on \(\mathbb{K}^{\mathbb{Z}}\) with nonzero weights has
the shadowing property.

For Problem~\ref{problem:B},
Theorem~\ref{thm:countable-products-no-metric-expansive} gives a negative
answer on countably infinite products.  More precisely, for every countably
infinite product \(X=\prod_{i\in I}E_i\) of nonzero
finite-dimensional normed spaces and every compatible metric on \(X\), no
operator in \(L(X)\) is metrically positively expansive and no operator in
\(GL(X)\) is metrically expansive.  The algebraic part is first proved for a
linear operator \(T_0\) on a countably infinite-dimensional vector space.  If
\(W=\operatorname{span}\{w_1,\ldots,w_s\}\), the finite sums
\(\sum_{j=1}^s a_j(T_0)w_j\), with \(a_j\in\mathbb{K}[t]\), generate the
\(\mathbb{K}[t]\)-submodule used in
Lemma~\ref{lem:bounded-orbit-functionals}.  When \(T_0\) is invertible, the
analogous sums with \(a_j\in\mathbb{K}[t,t^{-1}]\) generate the
\(\mathbb{K}[t,t^{-1}]\)-submodule used in
Lemma~\ref{lem:bounded-two-sided-orbit-functionals}.  The structure theorem for
finitely generated modules over principal ideal domains is used in these two
lemmas to obtain, respectively, a nonzero linear functional bounded on the
nonnegative iterates of \(W\), and a nonzero linear functional bounded on all
integer iterates of \(W\).
In Theorem~\ref{thm:countable-products-no-metric-expansive}, put
\[
V_F=\bigoplus_{i\in F}E_i^*\oplus \bigoplus_{i\in I\setminus F}\{0\}
\subset V:=\bigoplus_{i\in I}E_i^*
\]
for a nonempty finite set \(F\subset I\).  For \(T\in L(X)\),
Lemma~\ref{lem:bounded-orbit-functionals} is applied to \(T^*|_V\) and
\(W=V_F\).  For \(T\in GL(X)\),
Lemma~\ref{lem:bounded-two-sided-orbit-functionals} is applied to the bijective
operator \(T^*|_V\) and the same subspace \(V_F\).  The resulting functional
\(\Lambda\in V^*\) is represented by a nonzero \(x\in X\), in the sense that
\(\Lambda(\phi)=\phi(x)\) for \(\phi\in V\).  Hence, with
\(p_F(z)=\max_{i\in F}\|z_i\|\), the bounds on
\(\Lambda((T^*)^n\phi)\), \(\phi\in V_F\), give
\[
\sup_{n\in\mathbb{N}_0}p_F(T^n x)<+\infty
\quad\text{or}\quad
\sup_{n\in\mathbb{Z}}p_F(T^n x)<+\infty .
\]
A suitable scalar multiple of \(x\) violates the defining inequality for
metric positive expansivity or metric expansivity.  Applied to
\(\mathbb{K}^{\mathbb{Z}}\), this gives a negative answer to the
\(\mathbb{K}^{\mathbb{Z}}\)-part of Problem~\ref{problem:B}.

The paper is organized as follows.  Section~\ref{Sec:2} proves the shadowing
criterion and its application to \(H(\mathbb{C}^d)\).  Section~\ref{Sec:omega-shadowing}
proves shadowing for bilateral weighted shifts on \(\mathbb{K}^{\mathbb{Z}}\).
Section~\ref{Sec:3} proves the countable-product theorem excluding metric
positive expansivity for operators in \(L(X)\) and metric expansivity for
operators in \(GL(X)\).

\section[Shadowing on H(Cd)]{Shadowing on \(H(\mathbb{C}^d)\)}
\label{Sec:2}
A \textit{Fr\'{e}chet space} is a complete metrizable locally convex topological vector space.
Equivalently, it is a locally convex space whose topology is generated by a
countable family of seminorms.
Let $X$ be a Fr\'echet space and \(T\in GL(X)\). For
a neighborhood \(U\) of \(0\), a \textit{\(U\)-pseudotrajectory} of \(T\)
is a sequence \((x_n)_{n\in\mathbb{Z}}\) such that
\[
x_{n+1}-Tx_n\in U
\quad (\forall n\in\mathbb{Z}).
\]
According to
Bernardes et al.~\cite{BCDFP}, the operator \(T\in GL(X)\) has the
\textit{shadowing property}
if, for any neighborhood \(V\) of \(0\), there exists a neighborhood \(U\) of
\(0\) such that every \(U\)-pseudotrajectory \((x_n)_{n\in\mathbb{Z}}\) of
$T$ is $V$-shadowed by some $x\in X$, i.e.,
\[
x_n-T^n x\in V
\quad(\forall n\in\mathbb{Z}).
\]

\begin{proposition}\label{prop:shadowing-sufficient-condition}
Let $X$ be a Fr\'echet space and \(T\in GL(X)\).
Suppose that for each continuous seminorm \(p\) on \(X\), there exist
a continuous seminorm \(q\) on \(X\) and summable sequences
\((a_k)_{k\in\mathbb{N}}, (b_k)_{k\in\mathbb{N}}\) such that
\begin{enumerate}[{\rm (i)}]
\item
\(
p(T^k x)\le a_k q(x)
\quad(\forall k\in\mathbb{N},\ \forall x\in X);
\)
\item
For the same seminorm \(q\) and each continuous seminorm \(r\) on \(X\),
there exist \(N_r\in\mathbb{N}\) and a summable sequence
\((b_k^{(r)})_{k\ge N_r}\) such that
\[
r(T^k x)\le b_k^{(r)}q(x)
\quad(\forall k\ge N_r,\ \forall x\in X).
\]
\end{enumerate}
Then \(T\) has the shadowing property.
\end{proposition}

\begin{proof}
Fix any neighborhood $V$ of \(0\).
By the definition of the topology generated by seminorms,
there exist a continuous seminorm \(p\) on \(X\) and
\(\varepsilon>0\) such that
\[
\{x\in X : p(x)<\varepsilon\}\subset V.
\]
Let \(q\) and \((a_k)_{k\in\mathbb{N}}\) be the seminorm and
summable sequence given by condition {\rm (i)} for the seminorm \(p\).  Define
\begin{equation}\label{eq:summ_a_k}
\alpha=\sum_{k=1}^{\infty}a_k .
\end{equation}
Choose \(\delta>0\) such that
\(\delta(\alpha+1)<\varepsilon\) and set
\[
U=\{x\in X:p(x)<\delta\ \text{and}\ q(x)<\delta\}.
\]
Clearly, $U$ is a neighborhood of $0$.
For each \(U\)-pseudotrajectory \((x_n)_{n\in\mathbb{Z}}\)
of \(T\), let
\begin{equation}\label{eq:def_e_n}
e_n=x_{n+1}-Tx_n \quad (\forall n\in \mathbb{Z}).
\end{equation}
Then \(e_n\in U\), and hence by the definition of $U$,
\begin{equation}\label{eq:q_en_small}
p(e_n)<\delta\quad\text{and}\quad q(e_n)<\delta
\quad(\forall n\in\mathbb{Z}).
\end{equation}

For each \(n\in\mathbb{Z}\), we claim the series
\(\sum_{k=0}^{\infty}T^k e_{n-k-1}\)
converges in $X$.

\medskip

Indeed, for any continuous seminorm \(r\) on \(X\),
by condition {\rm (ii)} applied to
\(r\) and to the seminorm \(q\) chosen from condition {\rm (i)},
there exist \(N\in\mathbb{N}\) and a summable sequence
\((b_k)_{k\ge N}\) of non-negative numbers such that
\[
r(T^k x)\le b_k\cdot q(x)
\quad(\forall k\ge N,\ \forall x\in X).
\]
For any \(M,L\in\mathbb{N}\) with \(M\ge L\ge N\), we have
\[
r\bigg(\sum_{k=L}^{M}T^k e_{n-k-1}\bigg)
\le
\sum_{k=L}^{M}b_k\cdot q(e_{n-k-1})
<
\delta\sum_{k=L}^{M}b_k .
\]
Since \((b_k)_{k\ge N}\) is summable, the preceding estimate implies that the partial sums
\[
\sum_{k=0}^{M}T^k e_{n-k-1}
\quad(\forall M\in\mathbb{N}_0)
\]
of the series
\(
\sum_{k=0}^{\infty}T^k e_{n-k-1}
\) are Cauchy with respect to \(r\).  Since \(r\) is arbitrary, these partial sums
are Cauchy for the Fr\'echet topology of \(X\).
The completeness of \(X\) implies convergence of the series
\(
\sum_{k=0}^{\infty}T^k e_{n-k-1}
\) in \(X\).

Denote their limit by
\(d_n\):
\[
d_n=\sum_{k=0}^{\infty}T^k e_{n-k-1}.
\]
Since \(n\) is arbitrary, each \(d_n\) is well defined.
For each \(n\in\mathbb{Z}\), applying condition {\rm (i)}, \eqref{eq:summ_a_k},
and \eqref{eq:q_en_small} gives, for each
\(M\in\mathbb{N}\),
\[
p\bigg(\sum_{k=0}^{M}T^k e_{n-k-1}\bigg)
\le
p(e_{n-1})+\sum_{k=1}^{M}a_k\cdot q(e_{n-k-1})
\le
\delta(1+\alpha) .
\]
Letting \(M\to\infty\) gives \(p(d_n)\le \delta(1+\alpha)<\varepsilon\)
($n\in \mathbb{Z}$).
Since \(T\) is continuous, applying \(T\) termwise to the convergent series
defining \(d_n\) gives
\[
Td_n+e_n
=
\sum_{k=1}^{\infty}T^k e_{n-k}+e_n
=d_{n+1}.
\]
Thus the sequence \((d_n)_{n\in\mathbb{Z}}\) satisfies
\[
d_{n+1}=Td_n+e_n
\quad(\forall n\in\mathbb{Z}).
\]
Let \(x=x_0-d_0\) and
\[
h_n=x_n-T^n x
\quad(\forall n\in\mathbb{Z}).
\]
Then \(h_0=d_0\), and by \eqref{eq:def_e_n}
\[
Th_n+e_n=Tx_n-T^{n+1}x+e_n=(Tx_n+e_n)-T^{n+1}x=x_{n+1}-
T^{n+1}x=h_{n+1}
\quad(\forall n\in\mathbb{Z}),
\]
i.e.,
\[
h_{n+1}=Th_n+e_n
\quad(\forall n\in\mathbb{Z}).
\]
In particular,
\(d_{n+1}=Td_n+e_n\),
\(h_{n+1}=Th_n+e_n\),
and $h_0=d_0$. Since \(T\) is
invertible, we obtain \(h_n=d_n\) for all
\(n\in\mathbb{Z}\), implying
\[
x_n-T^n x=d_n\in V
\quad(\forall n\in\mathbb{Z}).
\]
Thus, \((x_n)_{n\in\mathbb{Z}}\) is
\(V\)-shadowed by the orbit of \(x\).
Since \(V\) is arbitrary,
\(T\) has the shadowing property.
\end{proof}

Let $\Omega\subset \mathbb{C}^d$ be
a nonempty open subset of $\mathbb{C}^d$ ($d\in \mathbb{N}$)
and $H(\Omega)$ a vector space of entire functions
\(f:\Omega\to\mathbb{K}\). Endow \(H(\Omega)\) with the
compact-open topology generated by the seminorms
\[
p_K(f)=\sup_{t\in K}|f(t)|
\quad(f\in H(\Omega),\ K\subset\Omega\text{ nonempty and compact}).
\]
Then, $H(\Omega)$ is a Fr\'{e}chet space with respect to this topology.

Fix a entire homeomorphism \(\psi:\Omega\to\Omega\) and a zero-free
entire function \(\varphi:\Omega\to\mathbb{K}\).
Define $T: H(\Omega)\to H(\Omega)$ by
\begin{equation}\label{eq:weighted-composition-operator}
Tf=\varphi\cdot(f\circ\psi).
\end{equation}
Then \(T\in GL(H(\Omega))\),
with inverse given by
\[
T^{-1}f
=
(\varphi\circ\psi^{-1})^{-1}\cdot(f\circ\psi^{-1}).
\]
Meanwhile, by direct calculation, we have
\begin{equation}\label{eq:Iterated-weighted-composition-operator}
T^kf(x)=\bigg(\prod_{j=0}^{k-1}\varphi(\psi^j(x))\bigg)\cdot f(\psi^k(x))\quad
(\forall k\in \mathbb{N},\
\forall f\in H(\Omega),\ \forall x\in \Omega).
\end{equation}


\begin{proposition}\label{prop:compact-open-shadowing}
Assume that for each compact set \(K\subset\Omega\), there exist
a compact set \(K_0\subset\Omega\) and a summable sequence
\((\alpha_n)_{n\in\mathbb{N}}\) such that:
\begin{enumerate}[{\rm (a)}]
\item For each \(k\in\mathbb{N}\),
\[
\psi^k(K)\subset K_0
\text{ and }
\sup_{x\in K}\prod_{j=0}^{k-1}|\varphi(\psi^j(x))|
\le \alpha_k;
\]

\item
For each compact set \(M\subset\Omega\), there exist
\(N\in\mathbb{N}\) and a summable sequence
\((\beta_k)_{k\ge N}\) such that, for all \(k\ge N\),
\[
\psi^k(M)\subset K_0
\text{ and }
\sup_{x\in M}\prod_{j=0}^{k-1}|\varphi(\psi^j(x))|
\le \beta_k .
\]
\end{enumerate}
Then, the operator \(T\) defined
by \eqref{eq:weighted-composition-operator} has the shadowing property.
\end{proposition}

\begin{proof}
We verify conditions (i) and (ii) of Proposition~\ref{prop:shadowing-sufficient-condition}.

\medskip

(i) For each continuous seminorm $p$ on \(H(\Omega)\), since the compact-open seminorms
generate the topology of \(H(\Omega)\), there exist compact sets
\(K_1,\ldots,K_s\subset\Omega\) and a constant \(C>0\) such that
\[
p(f)\le C\cdot \max_{1\le j\le s}p_{K_j}(f)
\quad(\forall f\in H(\Omega)).
\]
Set \(K=K_1\cup\cdots\cup K_s\).
Then, \(K\) is a compact subset of $\Omega$ and
\begin{equation}\label{eq:p(f)-up-bound}
p(f)\le C\cdot p_K(f)
\quad(\forall f\in H(\Omega)).
\end{equation}
From Condition (a), it follows that there exist a compact
set \(K_0\subset\Omega\) and a summable sequence
\((\alpha_k)_{k\in\mathbb{N}}\) such that
\begin{equation}\label{eq:Condition-a}
\psi^k(K)\subset K_0
\text{ and }
\sup_{x\in K}\prod_{j=0}^{k-1}|\varphi(\psi^j(x))|
\le \alpha_k \quad (\forall k\in \mathbb{N}).
\end{equation}
Take a continuous seminorm \(q=p_{K_0}.\)
The inclusion $\psi^n(K)\subset K_0$ in \eqref{eq:Condition-a}
implies that
\[
p_{\psi^k(K)}(f)\le p_{K_0}(f)=q(f)
\quad(\forall f\in H(\Omega),\ \forall k\in \mathbb{N}).
\]
This, together with \eqref{eq:Iterated-weighted-composition-operator},
\eqref{eq:p(f)-up-bound}, and \eqref{eq:Condition-a}, implies that,
for all \(k\in\mathbb{N}\) and all \(f\in H(\Omega)\),
\[
p(T^k f)
\le C\cdot p_{K}(T^k f)\leq
C\cdot \sup_{x\in K}\prod_{j=0}^{k-1}|\varphi(\psi^j(x))|
\cdot p_{\psi^k(K)}(f)
\le
C\cdot\alpha_k\cdot q(f).
\]

(ii) For each continuous seminorm \(r\) on \(H(\Omega)\), since the compact-open
seminorms generate the topology of \(H(\Omega)\), there exist compact sets
\(M_1,\ldots,M_m\subset\Omega\) and \(D>0\) such that
\[
r(f)\le D\cdot \max_{1\le j\le m}p_{M_j}(f)
\quad(\forall f\in H(\Omega)).
\]
Set \(M=M_1\cup\cdots\cup M_m\).  Then \(M\) is a compact subset of \(\Omega\)
and
\begin{equation}\label{eq:r(f)-up-bound}
r(f)\le D\cdot p_M(f)
\quad(\forall f\in H(\Omega)).
\end{equation}
With the compact set \(K_0\) fixed above and the compact
set $M$, from Condition (b), it follows that there exist
\(N\in\mathbb{N}\) and a summable sequence
\((\beta_k)_{k\ge N}\) such that
\begin{equation}\label{eq:Condition-b}
\psi^k(M)\subset K_0
\text{ and }
\sup_{x\in M}\prod_{j=0}^{k-1}|\varphi(\psi^j(x))|\le\beta_k
\quad (\forall k\ge N).
\end{equation}
Meanwhile, the inclusion
\(\psi^k(M)\subset K_0\) in \eqref{eq:Condition-b} gives
\[
p_{\psi^k(M)}(f)\le p_{K_0}(f)=q(f)
\quad(\forall f\in H(\Omega),\
\forall k\geq N).
\]
This, together with \eqref{eq:r(f)-up-bound} and \eqref{eq:Condition-b},
implies that,
for all \(k\geq N\) and all \(f\in H(\Omega)\),
\[
r(T^k f)
\le D\cdot p_{M}(T^k f)
\le
D\cdot \sup_{x\in M}\prod_{j=0}^{k-1}|\varphi(\psi^j(x))|\cdot p_{\psi^k(M)}(f)
\le
D\cdot \beta_k\cdot q(f).
\]

Therefore, \(T\) has the shadowing property by
Proposition~\ref{prop:shadowing-sufficient-condition}.
\end{proof}

For \(R>0\), define
\[
\overline{B}_R=\{x\in\mathbb{C}^d:\|x\|\le R\},\quad
p_R(f)=\max_{x\in\overline B_R}|f(x)|.
\]

\begin{theorem}\label{thm:HCd-shadowing}
Let \(d\in\mathbb{N}\) and \(0<|c|<1\).  If
\(B\in \operatorname{GL}_d(\mathbb{C})\) has spectral radius less than \(1\),
then the operator $T: H(\mathbb{C}^{d})\to H(\mathbb{C}^{d})$ defined by
\[
Tf(z)=c f(Bz)
\quad(f\in H(\mathbb{C}^d))
\]
has the shadowing property.
\end{theorem}

\begin{proof}
We verify conditions (a) and (b) of
Proposition~\ref{prop:compact-open-shadowing} for
\[
\Omega=\mathbb{C}^d,
\quad
H(\Omega)=H(\mathbb{C}^d),
\quad
\psi(x)=Bx,
\quad
\varphi(x)=c .
\]
(a) Since the spectral radius of \(B\) is less than \(1\), the powers \(B^n\)
converge to \(0\) in the operator norm on \(\mathbb{C}^d\). Then, for each compact set
\(K\subset\Omega\), we have
\[
\sup_{n\in\mathbb{N}_0}\sup_{x\in K}\|B^nx\|
\leq \sup_{n\in\mathbb{N}_0}\sup_{x\in K}\|B^n\|\cdot \|x\|<+\infty,
\]
and thus, there exists \(R>0\) such that
\[
B^n(K)\subset \overline{B}_R
\quad (\forall n\in \mathbb{N}).
\]
This, together with \(\psi(x)=Bx\), implies
\[
\psi^n(K)=
B^n(K)\subset
\overline{B}_R
\quad (\forall n\in \mathbb{N}).
\]
For the weight products,
\[
\sup_{x\in K}\prod_{j=0}^{n-1}|\varphi(\psi^j(x))|=|c|^n
\quad(\forall n\in\mathbb{N}),
\]
and \((|c|^n)_{n\in\mathbb{N}}\) is a summable sequence.  Thus Condition
{\rm (a)} of Proposition~\ref{prop:compact-open-shadowing} holds for \(K\) with
\(K_0=\overline{B}_R\) and \(\alpha_n=|c|^n\).

\medskip

(b) For each compact subset \(M\subset \Omega\), noting that \(M\) is bounded,
since the powers \(B^n\) converge to \(0\) in operator norm, there exists \(N\in\mathbb{N}\)
such that \(B^n(M)\subset \overline{B}_R\) for all \(n\ge N\).  Hence
\[
\psi^n(M)=B^n(M)\subset \overline{B}_R
\quad(\forall n\ge N).
\]
Meanwhile, for any \(n\ge N\),
\[
\sup_{x\in M}\prod_{j=0}^{n-1}|\varphi(\psi^j(x))|=|c|^n .
\]
The sequence \((|c|^n)_{n\ge N}\) is summable.  Thus Condition {\rm (b)} of
Proposition~\ref{prop:compact-open-shadowing} holds with the same compact set
\(K_0=\overline{B}_R\) and with
\(\beta_n=|c|^n\) for any \(n\ge N\).

Therefore, \(T\) has the shadowing property by
Proposition~\ref{prop:compact-open-shadowing}.
\end{proof}

\begin{remark}
For $c, \lambda\in \mathbb{C}$ with $0<|c|,~ |\lambda|<1$, define
$T: H(\mathbb{C})\to H(\mathbb{C})$ by
\[
Tf(x)=c\cdot f(\lambda x) \quad (f\in H(\mathbb{C}), \ x\in \mathbb{C}).
\]
It follows directly from Theorem~\ref{thm:HCd-shadowing} that
$T\in GL(H(\mathbb{C}))$ has the shadowing property. This gives an
affirmative answer to the \(H(\mathbb{C})\)-part of Problem~\ref{problem:A},
and the same construction applies on
\(H(\mathbb{C}^d)\) for each \(d\in\mathbb{N}\).
\end{remark}

\section[Shadowing on omega(Z)]{Shadowing on \(\mathbb{K}^{\mathbb{Z}}\)}
\label{Sec:omega-shadowing}

Let \(\mathbb{K}^{\mathbb{Z}}\)
be endowed with the product topology. For each \(m\in\mathbb{N}_0\), define
the continuous seminorm \(p_m\) on \(\mathbb{K}^{\mathbb{Z}}\) by
\[
p_m(x)=\max_{|j|\le m}|x_j|.
\]
The product topology coincides with the locally convex topology generated
by the seminorms \((p_m)_{m\in \mathbb{N}_0}\).  A neighborhood base at
\(0\) is therefore given by
\begin{equation}\label{eq:local_base_0}
\{\{x\in\mathbb{K}^{\mathbb{Z}}:p_m(x)<\varepsilon\}:
m\in\mathbb{N}_0,\ \varepsilon>0\}.
\end{equation}

For a sequence \(w=(w_i)_{i\in\mathbb{Z}}\) of nonzero scalars, define the
\textit{bilateral weighted forward shift} \(F_w:\mathbb{K}^{\mathbb{Z}}\to\mathbb{K}^{\mathbb{Z}}\)
by
\begin{equation}\label{eq:Def-bilateral_weighted_forward_shift}
(F_w x)_j=w_{j-1}\cdot x_{j-1}\quad(j\in\mathbb{Z}).
\end{equation}
Then \(F_w\in GL(\mathbb{K}^{\mathbb{Z}})\).  Its inverse is
\[
(F_w^{-1}x)_j=\frac{x_{j+1}}{w_j}\quad(j\in\mathbb{Z}).
\]
Here, no boundedness assumption on \(w\) or \(1/w\) is required, since each coordinate
of \(F_wx\) and \(F_w^{-1}x\) is determined by a single coordinate of \(x\).
For each \(i\in\mathbb{Z}\) and each \(n\in\mathbb{N}\), define the weight product
\[
W_i(n)=w_i\cdot w_{i+1}\cdots w_{i+n-1},
\]
and \(W_i(0)=1\) (empty product).
Then, for all \(n\in\mathbb{N}_0\) and all \(j\in\mathbb{Z}\),
\begin{equation}\label{eq:Iteration_Formula}
(F_w^n x)_j=W_{j-n}(n)\cdot x_{j-n},
\quad
(F_w^{-n}x)_j=W_j(n)^{-1}\cdot x_{j+n}.
\end{equation}

In the proof of following
Theorem~\ref{thm:omega-weighted-shift-shadowing}, a
pseudotrajectory is written as \((x^{(n)})_{n\in\mathbb{Z}}\), and
\(x_j^{(n)}\) denotes the \(j\)-th coordinate of \(x^{(n)}\).

\begin{theorem}\label{thm:omega-weighted-shift-shadowing}
The bilateral weighted forward shift \(F_w\) defined by
\eqref{eq:Def-bilateral_weighted_forward_shift} with nonzero
weights has the shadowing property.
\end{theorem}

\begin{proof}
For each neighborhood $V_1$ of $0$, by \eqref{eq:local_base_0},
there exist $m\in\mathbb{N}_0$ and $\varepsilon>0$ such that
\[
V:=\{x\in\mathbb{K}^{\mathbb{Z}}:p_m(x)<\varepsilon\}\subset V_1.
\]
Let
\[
C_m^+
=
\max_{1\le j\le m}\sum_{q=1}^{j}|W_q(j-q)|,
\quad
C_m^-
=
\max_{-m\le j\le -1}
|W_j(-j)|^{-1}\cdot \sum_{q=j+1}^{0}|W_q(-q)|,
\]
and
\[
C_m=1+C_m^{+}+C_m^{-},
\quad
\xi=\frac{\varepsilon}{C_m},
\]
where the maximum over an empty set is taken to be \(0\). Choose
\[
U=\{x\in\mathbb{K}^{\mathbb{Z}}:p_m(x)<\xi\}.
\]
For each \(U\)-pseudotrajectory \((x^{(n)})_{n\in\mathbb{Z}}\)
of \(F_w\), put
\begin{equation}\label{eq:Def_e^{n}}
e^{(n)}=x^{(n+1)}-F_{w}x^{(n)}
\quad(n\in\mathbb{Z}).
\end{equation}
This, together with \eqref{eq:Def-bilateral_weighted_forward_shift}, implies
\begin{equation}\label{eq:Def_e_q^{n}}
e_q^{(n)}=x^{(n+1)}_q-(F_{w}x^{(n)})_{q}=
x^{(n+1)}_q-w_{q-1}\cdot x^{(n)}_{q-1}
\quad(n\in\mathbb{Z}).
\end{equation}
By the choice of \(U\), we have
\begin{equation}\label{eq:e_q-Upper-Bound}
|e_q^{(n)}|\leq p_{m}(x^{(n+1)}-F_{w}x^{(n)})<\xi
\quad(n\in\mathbb{Z},\ |q|\le m).
\end{equation}

Choose a sequence \(y\in\mathbb{K}^{\mathbb{Z}}\) as
\[
y_i=
\begin{cases}
W_i(-i)^{-1}\cdot x^{(-i)}_0, & i\le0,\\
W_0(i)\cdot x^{(-i)}_0, & i>0.
\end{cases}
\]
By direct calculation, we have that
\begin{itemize}
  \item For each \(n\geq 0\),
\[
(F_w^n y)_0=W_{-n}(n)\cdot y_{-n}=x^{(n)}_0;
\]
  \item For each \(n<0\),
\[
(F_w^n y)_0=W_0(-n)^{-1}\cdot y_{-n}=x^{(n)}_0;
\]
\end{itemize}
and thus
\begin{equation}\label{eq:F_w^n(y)_0}
(F_w^n y)_0=x^{(n)}_0
\quad(\forall n\in\mathbb{Z}).
\end{equation}

Fix \(n\in\mathbb{Z}\). For each $j\in \mathbb{Z}$ with \(|j|\le m\), to prove
$|(F_w^n y)_j-x^{(n)}_j|<\varepsilon$, we consider the following
three cases:

1)
If \(j>0\), then \(F_w^n y = F_w^j(F_w^{n-j}y)\).
Since \((F_w^{n-j}y)_0 = x^{(n-j)}_0\) by \eqref{eq:F_w^n(y)_0},
and since \((F_w^j x)_j = W_0(j)\cdot x_0\) for any
\(x\in \mathbb{K}^{\mathbb{Z}}\) by \eqref{eq:Iteration_Formula},
we obtain
\begin{equation}\label{eq:F_w^n(y)}
(F_w^n y)_j=W_0(j)\cdot x^{(n-j)}_0.
\end{equation}
Meanwhile, by \eqref{eq:Def_e^{n}}, we have
\begin{align*}
x^{(n)}=& F_{w}x^{(n-1)}+e^{(n-1)}\\
=&
\cdots \\
=& F_{w}^{j}x^{(n-j)}+
F_{w}^{j-1}e^{(n-j)}+\cdots +e^{(n-1)}\\
=& F_{w}^{j}x^{(n-j)}
+\sum_{q=1}^{j}F_{w}^{j-q}e^{(n-j+q-1)},
\end{align*}
and thus, by \eqref{eq:Iteration_Formula}
\[
x^{(n)}_j
=
W_0(j)\cdot x^{(n-j)}_0
+
\sum_{q=1}^{j} W_q(j-q)\cdot e_q^{(n-j+q-1)}.
\]
This, together with \eqref{eq:e_q-Upper-Bound} and \eqref{eq:F_w^n(y)},
implies
\[
\left|(F_w^n y)_j-x^{(n)}_j\right|
=\left|\sum_{q=1}^{j} W_q(j-q)\cdot e_q^{(n-j+q-1)}\right|
\leq \sum_{q=1}^{j}|W_q(j-q)|\cdot \xi
\leq C_{m}^{+}\cdot \xi
< C_{m}\cdot \xi=\varepsilon .
\]

2)
If \(j<0\), by \eqref{eq:Iteration_Formula},
\[
(F_w^{j}x)_j = W_j(-j)^{-1}\cdot x_0
\quad (\forall x\in\mathbb{K}^{\mathbb{Z}}).
\]
By \eqref{eq:F_w^n(y)_0}, taking \(x = F_w^{n-j} y\) gives
\begin{equation}\label{eq:F_w^n(y)-2}
(F_w^{n}y)_j = (F_w^{j}(F_w^{n-j} y))_j = W_j(-j)^{-1}\cdot(F_w^{n-j} y)_0
=W_j(-j)^{-1}\cdot x^{(n-j)}_0.
\end{equation}
Meanwhile, by \eqref{eq:Def_e^{n}}, we have
\begin{align*}
x^{(n-j)}=& F_{w}x^{(n-j-1)}+e^{(n-j-1)}\\
=&
\cdots \\
=& F_{w}^{-j}x^{(n)}+
F_{w}^{-j-1}e^{(n)}+\cdots +e^{(n-j-1)}\\
=& F_{w}^{-j}x^{(n)}
+\sum_{q=j+1}^{0}F_{w}^{-q}e^{(n+q-j-1)},
\end{align*}
and thus, by \eqref{eq:Iteration_Formula}
\[
x^{(n-j)}_0
=
W_j(-j)\cdot x^{(n)}_j
+
\sum_{q=j+1}^{0}W_q(-q)\cdot e_q^{(n+q-j-1)}.
\]
Then
\[
x^{(n)}_j=W_j(-j)^{-1}\cdot x^{(n-j)}_0
-
W_j(-j)^{-1}\cdot \sum_{q=j+1}^{0}W_q(-q)\cdot
e_q^{(n+q-j-1)}.
\]
This, together with \eqref{eq:e_q-Upper-Bound} and \eqref{eq:F_w^n(y)-2},
implies
\begin{align*}
|(F_w^n y)_j-x^{(n)}_j|
=& \left|W_j(-j)^{-1}\cdot \sum_{q=j+1}^{0}W_q(-q)\cdot
e_q^{(n+q-j-1)}\right|\\
\le &
|W_j(-j)|^{-1}\cdot \sum_{q=j+1}^{0}|W_q(-q)|\cdot \xi
\leq C_{m}^{-}\cdot \xi <\varepsilon .
\end{align*}

3) If \(j=0\), by \eqref{eq:F_w^n(y)_0}, then
\[
|(F_w^n y)_0-x^{(n)}_0|=0<\varepsilon.
\]

Thus, $|(F_w^n y)_j-x^{(n)}_j|<\varepsilon$ holds for all \(|j|\le m\).
By the definition of \(p_m\), we have
\[
p_m(F_w^n y-x^{(n)})<\varepsilon
\quad(\forall n\in\mathbb{Z}).
\]
Therefore, \(F_w^n y-x^{(n)}\in V\subset V_1\) for all \(n\in\mathbb{Z}\).
Hence \(F_w\) has the shadowing property.
\end{proof}

\section{Metric expansivity on countable products}
\label{Sec:3}
According to Bernardes et al.~\cite{BCDFP}, a continuous map \(T:X\to X\) on a
metric space \((X,d)\) is \textit{metrically positively expansive} if, there exists
\(\delta>0\) such that, for any \(x, y\in X\) with \(x\ne y\),
there exists \(n\in\mathbb{N}_0\) such that
\[
d(T^n x,T^n y)\ge\delta.
\]
An invertible map \(T\) is \textit{metrically expansive} if, there exists \(\delta>0\)
such that, for any \(x, y\in X\) with \(x\ne y\), there exists \(n\in\mathbb{Z}\)
such that
\[
d(T^n x,T^n y)\ge\delta.
\]

For a vector space \(X\) over \(\mathbb{K}\), let \(X^*\) be the space
of linear functionals on \(X\), i.e.,
\[
X^*:=\{\phi:X\to\mathbb{K}:\phi\ \text{is linear}\}
\]
which is called the \textit{dual space} of \(X\).
If \(T:X\to X\) is linear, its \textit{algebraic adjoint} is the linear map
\[
T^*:X^*\to X^*,\quad T^*\phi=\phi\circ T
\quad(\forall \phi\in X^*).
\]

For a family \((X_i)_{i\in I}\) of vector spaces, set
\[
\bigoplus_{i\in I}X_i=
\left\{(x_i)_{i\in I}\in\prod_{i\in I}X_i:
\#\{i\in I:x_i\ne0\}<+\infty\right\},
\]
which is called the \textit{algebraic direct sum} of the family \((X_i)_{i\in I}\)
(see \cite[Sec.~10.3, Exercise~20]{DF}).

Some arguments in this section rely on the structure theory
of finitely generated modules over principal ideal domains.
For completeness, we use the standard terminology on modules from Dummit and
Foote~\cite[Chapters~10 and 12]{DF}.

Let \(R\) be a (commutative) ring with identity $1$.
An \textit{\(R\)-module} is a set $M$ together with
\begin{enumerate}
  \item[(1)] a binary operation $+$ on $M$ under which $M$ is an
  abelian group, and
  \item[(2)] an action of $R$ on $M$ (that is, a map
  $R\times M\to M$) denoted by $rm$ ($\forall r\in R,\
  \forall m\in M$), which satisfies
  \begin{enumerate}
    \item[(i)] $(r+s)m=rm+sm \quad (\forall r, s\in R, \ \forall m\in M)$;
    \item[(ii)] $(rs)m=r(sm) \quad (\forall r, s\in R, \ \forall m\in M)$;
    \item[(iii)] $r(m+n)=rm+rn \quad (\forall r\in R, \ \forall m, n\in M)$;
    \item[(iv)] $1m=m\quad (\forall m\in M)$.
  \end{enumerate}
\end{enumerate}

Let \(R\) be a ring and $M$ an \(R\)-module.
An \textit{$R$-submodule} of $M$ is a subgroup $N$ of $M$ which is
closed under the action of ring elements, i.e., $rn\in N$,
for all $r\in R$, $n\in N$.

A submodule $N$ of $M$ is \textit{finitely generated} if there exists
finite subset $A$ of $M$ such that
\[
N=\{r_1a_1+r_2a_2+\cdots+r_sa_s:
s\in\mathbb{N},\ r_1, \ldots, r_s\in R,\
a_1, \ldots, a_s\in A\}.
\]

An ideal \(I\) of a ring \(R\) is \textit{principal} if \(I=Ra\) for some \(a\in R\).
A commutative ring with identity $1\neq 0$ is an \textit{integral domain}
if it has  no zero divisors. A \textit{Principal Ideal
Domain} (P.I.D.) is an integral domain in which every ideal is principal.  The
following Fundamental Theorem states that every finitely generated module
over a P.I.D. is isomorphic to the direct sum of finitely many cyclic modules.

\begin{lemma}
[{\textrm{\protect\cite[Chap.~12, Theorem~5]{DF}}}]
\label{thm:PID-module-structure}
Let \(R\) be a P.I.D. and \(M\) a finitely generated \(R\)-module.
Then
\[
M \;\cong\; R^{r}\oplus\bigoplus_{i=1}^{m}R\big/(a_i),
\]
for some \(r,m\in\mathbb{N}_0\) and nonzero nonunit elements
\(a_1, \ldots, a_m\in R\), where \((a_i)=Ra_i\).
\end{lemma}

\begin{lemma}\label{lem:laurent-polynomial-PID}
The Laurent polynomial ring \(\mathbb{K}[t,t^{-1}]\) is a P.I.D.
\end{lemma}

\begin{proof}
(1) Let
\[
0\ne f=\sum_{k=m}^{n}a_kt^k
\text{ and }
0\ne g=\sum_{\ell=r}^{s}b_\ell t^\ell
\]
be elements of \(\mathbb{K}[t,t^{-1}]\), where \(a_m\ne0\) and \(b_r\ne0\).
The coefficient of \(t^{m+r}\) in \(fg\) is \(a_mb_r\ne0\). Hence
\(\mathbb{K}[t,t^{-1}]\) is an integral domain.

\medskip

(2) Let \(J\) be an ideal of
\(\mathbb{K}[t,t^{-1}]\). To prove that \(J\) is a principal ideal of
\(\mathbb{K}[t,t^{-1}]\), we consider the following two cases:

2-1) If \(J=\{0\}\), it is clear that \(J\) is principal.

2-2) If \(J\ne\{0\}\), since \(J\) is an ideal of
\(\mathbb{K}[t,t^{-1}]\), then \(I:=J\cap\mathbb{K}[t]\) is an ideal of
\(\mathbb{K}[t]\), because
\[
\mathbb{K}[t]I
=
\mathbb{K}[t](J\cap\mathbb{K}[t])
\subset
(\mathbb{K}[t]J)\cap(\mathbb{K}[t]\mathbb{K}[t])
\subset
J\cap\mathbb{K}[t]
=I.
\]
By \(J\ne\{0\}\), choose \(h_1\in J\setminus \{0\}\). Then, there exists
\(N_1\in\mathbb{N}_0\) such that \(t^{N_1}h_1\in\mathbb{K}[t]\).  Since \(J\) is an
ideal of \(\mathbb{K}[t,t^{-1}]\), \(0\ne t^{N_1}h_1\in J\cap\mathbb{K}[t]=I\),
implying \(I\ne\{0\}\). Meanwhile, since \(\mathbb{K}[t]\) is a P.I.D.
(\cite[Sec.~8.1, Proposition and Sec.~9.2, Theorem~3]{DF}),
\begin{equation}\label{eq:I_PID}
\exists p\in\mathbb{K}[t]\setminus \{0\}
\text{  s.t. }
I=\mathbb{K}[t]p.
\end{equation}

We claim that \(J=\mathbb{K}[t,t^{-1}]p\). Indeed, from
\(p=1p\in \mathbb{K}[t]p\subset I\subset J\), it follows that
\[
\mathbb{K}[t,t^{-1}]p
\subset
\mathbb{K}[t,t^{-1}]J
\subset J.
\]
For the opposite inclusion, let \(h\in J\). Then there exists
\(N\in\mathbb{N}_0\) such that \(t^Nh\in\mathbb{K}[t]\), and thus
\[
t^Nh\in \mathbb{K}[t]\cap(\mathbb{K}[t,t^{-1}]J)
\subset
\mathbb{K}[t]\cap J
=I.
\]
This, together with \eqref{eq:I_PID}, implies
\(t^Nh=qp\) for some \(q\in\mathbb{K}[t]\). Therefore
\[
h=(t^{-N}q)p\in \mathbb{K}[t,t^{-1}]p.
\]
Hence \(J=\mathbb{K}[t,t^{-1}]p\), and \(\mathbb{K}[t,t^{-1}]\) is a
P.I.D.
\end{proof}

\begin{lemma}\label{lem:bounded-orbit-functionals}
Let \(T_0:V\to V\) be a linear operator on a countably infinite-dimensional
vector space \(V\) over \(\mathbb{K}\).  If \(W\) is a finite-dimensional linear
subspace of \(V\), then there exists a nonzero linear functional
\(\Lambda\in V^*\) such that
\[
\sup_{n\in\mathbb{N}_0}|\Lambda(T_0^n v)|<+\infty
\]
for every \(v\in W\).
\end{lemma}

\begin{proof}
We define the action of $\mathbb{K}[t]$ on
$V$ by
\[
a(t) v:=\sum_{k=0}^{n}a_k T_0^kv,
\]
for $a(t)=\sum_{k=0}^{n}a_k t^k\in \mathbb{K}[t]$
($n\in \mathbb{N}_0$). Then \(V\) is a \(\mathbb{K}[t]\)-module under this action
and the vector addition on \(V\) (see \cite[Sec.~10.1, Example: \(F\lbrack x\rbrack\)-modules]{DF}).

Since $W$ is a finite-dimensional linear subspace
of $V$, fix a basis
\(\{w_1,\ldots,w_s\}\)
of \(W\), and let
\[
M_W=\sum_{j=1}^s\mathbb{K}[t] w_j.
\]
Then $M_{W}$ is a finitely generated
\(\mathbb{K}[t]\)-submodule.
We consider the following two cases:

(1) If \(M_W\ne V\), then \(V/M_W\ne0\).
Choose a nonzero linear functional
\(\Lambda_0\in (V/M_W)^*\), and define
\(\Lambda: V\to\mathbb{K}\) by
\begin{equation}\label{eq:Lamda-Def-LP}
\Lambda(v)=\Lambda_0(v+M_W)
\quad (v\in V).
\end{equation}
Then \(\Lambda\in V^*\), and
$\Lambda\neq 0$ because $\Lambda_0\neq 0$.
For any \(v\in W\) and any \(n\in\mathbb{N}_0\), one has
\[
T_0^n v=t^n v\in M_W,
\]
which, together with \eqref{eq:Lamda-Def-LP}, implies
\[
\Lambda(T_0^n v)=
\Lambda_0(T_0^n v+M_W)
=\Lambda_0(0+M_W)
=0.
\]
Thus,
\[
\sup_{n\in\mathbb{N}_0}
|\Lambda(T_0^n v)|
=0 < +\infty \quad
(\forall v\in W).
\]

(2)
If \(M_W=V\), then
\(V=M_W=\sum_{j=1}^s\mathbb{K}[t]w_j,\)
implying that \(V\) is a finitely generated
\(\mathbb{K}[t]\)-module.
Since \(\mathbb{K}[t]\) is a P.I.D.
(\cite[Sec.~8.1, Proposition and Sec.~9.2, Theorem~3]{DF}),
applying Lemma~\ref{thm:PID-module-structure} yields
\[
V
\cong
\mathbb{K}[t]^{\,r}
\oplus
\bigoplus_{i=1}^{m}\mathbb{K}[t]\big/(a_i(t)),
\]
where \(r\in\mathbb{N}_0\) and
\(a_1(t),\ldots,a_m(t)\) are nonzero nonunit elements of \(\mathbb{K}[t]\).

We claim that \(r>0\).
Indeed, if \(r=0\), then
\(
V
\cong
\bigoplus_{i=1}^{m}\mathbb{K}[t]\big/(a_i(t)).
\)
Since $0\neq a_i(t)\in \mathbb{K}[t]$, by the division algorithm
(\cite[Sec.~9.2, Theorem~3]{DF}), for any $f(t)\in \mathbb{K}[t]$,
there exist unique $q(t), s(t)\in \mathbb{K}[t]$ such that
\[
f(t)=q(t)\cdot a_i(t)+s(t),
\text{ with }
s(t)=0\text{ or }
\mathrm{degree}(s(t))<\mathrm{degree}(a_i(t)),
\]
and thus \(f(t)+(a_i(t))=s(t)+(a_i(t))\).
Hence,
\[
\dim_{\mathbb{K}}\mathbb{K}[t]\big/(a_i(t))
\leq \mathrm{degree}(a_i(t))
<+\infty
\quad(1\le i\le m),
\]
(also see \cite[Sec. 11.1, Page~411, Example~(2)]{DF}).
Consequently,
\[
\dim_{\mathbb{K}} V
=\sum_{i=1}^{m}
\dim_{\mathbb{K}}\mathbb{K}[t]\big/(a_i(t))
<+\infty,
\]
contradicting the assumption that \(V\) is countably infinite-dimensional.
Therefore, \(r>0\).

Fix an \(\mathbb{K}[t]\)-module isomorphism
\[
\Phi:V\to (\mathbb{K}[t])^{\,r}
\oplus
\bigoplus_{i=1}^{m}\mathbb{K}[t]\big/(a_i(t)).
\]
For each \(v\in V\), write
\[
\Phi(v)=
\big(P_1(v),\ldots,P_r(v),
\overline{Q}_1(v),\ldots,\overline{Q}_m(v)\big),
\]
where \(P_j(v)\in\mathbb{K}[t]\) and
\(\overline{Q}_i(v)\in\mathbb{K}[t]\big/(a_i(t))\).
Define
\[
\pi:V\to\mathbb{K}[t],
\quad
\pi(v)=P_1(v).
\]
Since \(\Phi\) is an \(\mathbb{K}[t]\)-module isomorphism,
for $t^{n}\in \mathbb{K}[t]$ ($n\in \mathbb{N}_0$), we have
\[
\Phi(t^n v)=t^n\Phi(v),
\]
and thus the definition of \(\pi\) gives
\begin{equation}\label{eq:R-module-isomorphism}
\pi(t^n v)
=
t^n\pi(v)
\quad(\forall n\in\mathbb{N}_0,\ \forall v\in V).
\end{equation}

Define a \(\mathbb{K}\)-linear functional
\[
\sigma:\mathbb{K}[t]\to\mathbb{K},
\quad
\sigma(P)=P(1)
\quad(P\in\mathbb{K}[t]),
\]
and define
\[
\Lambda=\sigma\circ\pi.
\]
Since \(\pi\) and \(\sigma\) are \(\mathbb{K}\)-linear, \(\Lambda\in V^*\).
Since \(r>0\), the map \(\pi\) is surjective.
This, together with \(\sigma(1)=1\ne0,\)
implies \(\Lambda\ne0.\)
Moreover, for any
\(v\in W\) and any \(n\in\mathbb{N}_0\),
by \eqref{eq:R-module-isomorphism}, we get
\[
\Lambda(T_0^nv)
=\sigma(\pi(t^n v))
=\sigma(t^n\cdot \pi(v))
=\pi(v)(1).
\]
Hence,
\[
\sup_{n\in\mathbb{N}_0}
|\Lambda(T_0^nv)|
=
|\pi(v)(1)|
<
+\infty \quad (\forall v\in W).
\]
\end{proof}

For bijective linear operators, metric expansivity involves both positive and negative
iterates.  The corresponding algebraic estimate is the following
Laurent-polynomial version.

\begin{lemma}\label{lem:bounded-two-sided-orbit-functionals}
Let \(T_0:V\to V\) be a bijective linear operator on a countably
infinite-dimensional vector space \(V\) over \(\mathbb{K}\).
If \(W\) is a finite-dimensional linear subspace of
\(V\), then there exists a nonzero linear functional \(\Lambda\in V^*\) such that
\[
\sup_{n\in\mathbb{Z}}|\Lambda(T_0^n v)|<+\infty
\]
for every \(v\in W\).
\end{lemma}

\begin{proof}
Since \(T_0\) is bijective on $V$,
we define the action of \(\mathbb{K}[t,t^{-1}]\) on \(V\) by
\[
a(t)v:=\sum_{k=p}^{q} a_k T_0^k v,
\]
for \(a(t)=\sum_{k=p}^{q}a_k t^k\in \mathbb{K}[t,t^{-1}]\)
(\(p,q\in\mathbb{Z}\) and \(p\le q\)), where $\mathbb{K}[t, t^{-1}]$
is the Laurent polynomial ring.  Then \(V\) is an
\(\mathbb{K}[t,t^{-1}]\)-module under this action and the vector addition on \(V\).

Since
\(W\) is a finite-dimensional linear subspace of
\(V\), fix a basis \(\{w_1,\ldots,w_s\}\)
of \(W\), and let
\[
M_W = \sum_{j=1}^s\mathbb{K}[t,t^{-1}]w_j.
\]
Then $M_{W}$ is a finitely generated
\(\mathbb{K}[t, t^{-1}]\)-submodule.
We consider the following two cases:

(1)
If \(M_W\ne V\), then \(V/M_W\ne0\).
Choose a nonzero linear functional
\(\Lambda_0\in (V/M_W)^*\), and define
\(\Lambda: V\to\mathbb{K}\) by
\[
\Lambda(v)
= \Lambda_0(v+M_W)
\quad (v\in V).
\]
Then \(\Lambda\in V^*\), and
$\Lambda\neq 0$ because $\Lambda_0\neq 0$.
For any \(v\in W\) and any \(n\in\mathbb{Z}\), one has
\[
T_0^n v=t^n v\in M_W,
\]
which implies
\[
\Lambda(T_0^n v)
= \Lambda_0(T_0^n v+M_W)
= \Lambda_0(0+M_W)
= 0.
\]
Thus,
\[
\sup_{n\in\mathbb{Z}}
|\Lambda(T_0^n v)|
= 0 < +\infty
\quad (\forall\,v\in W).
\]

\medskip

(2)
If \(M_W=V\), then \(V=M_W=\sum_{j=1}^s\mathbb{K}[t,t^{-1}]w_j,\)
implying that \(V\) is a finitely generated
\(\mathbb{K}[t,t^{-1}]\)-module.
Since \(\mathbb{K}[t,t^{-1}]\) is a P.I.D. by
Lemma~\ref{lem:laurent-polynomial-PID}, applying
Lemma~\ref{thm:PID-module-structure} yields
\[
V \cong
(\mathbb{K}[t,t^{-1}])^{r}
\oplus \bigoplus_{i=1}^{m}\mathbb{K}[t,t^{-1}]\big/(a_i(t)),
\]
where
\(r, m\in\mathbb{N}_0\)
and
\(a_1(t),\ldots,a_m(t)\) are nonzero nonunit
elements of \(\mathbb{K}[t,t^{-1}]\).

\medskip

We claim that \(r>0\). Indeed, if \(r=0\), then
\(V \cong \bigoplus_{i=1}^{m}\mathbb{K}[t,t^{-1}]\big/(a_i(t)).\)
Since \(0\ne a_i(t)\in\mathbb{K}[t,t^{-1}]\), there exist
\(\ell_i\in\mathbb{Z}\) and \(b_i(t)\in\mathbb{K}[t]\setminus\{0\}\)
with \(b_i(0)\ne0\) such that
\(a_i(t)=t^{\ell_i}b_i(t),\) and thus
\((a_i(t))=(b_i(t))\).
Moreover, from \(b_i(0)\ne0\), it follows that \(t\) and \(b_i(t)\) are relatively prime
in \(\mathbb{K}[t]\), implying that there exist
\(u_i(t),v_i(t)\in\mathbb{K}[t]\) such that
\(
u_i(t)\cdot t+v_i(t)\cdot b_i(t)=1.
\)
Thus, for any $n\in \mathbb{N}$,
\begin{equation*}
(u_i(t)\cdot t)^n=(1-v_i(t)\cdot b_i(t))^n
=1+\sum_{j=1}^n\binom{n}{j}(-v_i(t)\cdot b_i(t))^{j},
\end{equation*}
implying
\begin{equation}\label{eq:prime-identy-2}
(u_i(t)\cdot t)^n-1\in (b_i(t)).
\end{equation}
For any \(f(t)\in\mathbb{K}[t,t^{-1}]\), there exist \(N\in\mathbb{N}_0\) and
\(h(t)\in\mathbb{K}[t]\) such that \(f(t)=t^{-N}\cdot h(t)\).
Together with \eqref{eq:prime-identy-2},
we have
\[
(u_i(t)^N\cdot h(t)-f(t))\cdot t^{N}=(u_i(t)^N-t^{-N})\cdot t^N\cdot h(t)
=(u_i(t)^N\cdot t^{N}-1)\cdot h(t)\in (b_i(t)),
\]
and thus
\[
u_i(t)^N\cdot h(t)-f(t)\in (b_i(t)),
\]
which implies
\[
f(t)+(b_i(t))
=u_i(t)^N\cdot h(t)+(b_i(t)).
\]

Applying the division algorithm in \(\mathbb{K}[t]\)
(\cite[Sec.~9.2, Theorem~3]{DF}) to \(u_i(t)^Nh(t)\), each coset of
\(\mathbb{K}[t,t^{-1}]\big/(b_i(t))\) has a representative
\(s(t)\in\mathbb{K}[t]\) with \(s(t)=0\) or
\(\mathrm{degree}(s(t))<\mathrm{degree}(b_i(t))\). Hence,
\[
\dim_{\mathbb{K}}\mathbb{K}[t,t^{-1}]\big/(a_i(t))
=
\dim_{\mathbb{K}}\mathbb{K}[t,t^{-1}]\big/(b_i(t))
\le\mathrm{degree}(b_i(t))
<+\infty
\quad
(1\le i\le m).
\]
Consequently,
\[
\dim_{\mathbb{K}}V = \sum_{i=1}^{m}
\dim_{\mathbb{K}}\mathbb{K}[t,t^{-1}]\big/(a_i(t))
< +\infty,
\]
contradicting the assumption that \(V\) is countably infinite-dimensional.
Therefore, \(r>0\).

Fix an \(\mathbb{K}[t,t^{-1}]\)-module isomorphism
\[
\Phi:V\to (\mathbb{K}[t,t^{-1}])^{r}
\oplus
\bigoplus_{i=1}^{m}\mathbb{K}[t,t^{-1}]\big/(a_i(t)).
\]
For each \(v\in V\), write
\[
\Phi(v)= \big(P_1(v),\ldots,P_r(v),
\overline{Q}_1(v),\ldots,\overline{Q}_m(v)\big),
\]
where \(P_j(v)\in\mathbb{K}[t,t^{-1}]\) and
\(\overline{Q}_i(v)\in\mathbb{K}[t,t^{-1}]\big/(a_i(t))\).
Define
\[
\pi:V\to\mathbb{K}[t,t^{-1}],
\quad
\pi(v)=P_1(v).
\]
Since \(\Phi\) is an \(\mathbb{K}[t,t^{-1}]\)-module isomorphism,
for $t^{n}\in \mathbb{K}[t,t^{-1}]$ ($n\in \mathbb{Z}$), we have
\[
\Phi(t^n v)=t^n\Phi(v),
\]
and thus the definition of \(\pi\) gives
\begin{equation}\label{eq:R-module-isomorphism-2}
\pi(t^n v)
=
t^n\cdot \pi(v)
\quad(\forall n\in\mathbb{Z},\ \forall v\in V).
\end{equation}

Define a \(\mathbb{K}\)-linear functional
\[
\sigma:\mathbb{K}[t,t^{-1}]\to\mathbb{K},
\quad
\sigma(P)=P(1)
\quad(P\in\mathbb{K}[t,t^{-1}]),
\]
and define
\[
\Lambda=\sigma\circ\pi.
\]
Since \(\pi\) and \(\sigma\) are \(\mathbb{K}\)-linear, \(\Lambda\in V^*\).
Since \(r>0\), the map \(\pi\) is surjective.
This, together with \(\sigma(1)=1\ne0,\)
implies \(\Lambda\ne0.\)
Moreover, for any
\(v\in W\) and any \(n\in\mathbb{Z}\),
by \eqref{eq:R-module-isomorphism-2}, we get
\[
\Lambda(T_0^nv)=
\sigma(\pi(t^n v))
=\sigma(t^n\cdot \pi(v))
=\pi(v)(1).
\]
Hence,
\[
\sup_{n\in\mathbb{Z}}
|\Lambda(T_0^nv)|
= |\pi(v)(1)| <
+\infty \quad (\forall v\in W).
\]
\end{proof}

\begin{lemma}
\label{lem:dual-of-algebraic-direct-sum}
Let \((X_i)_{i\in I}\) be a family of vector spaces over \(\mathbb K\).
Then, for each
\(\Lambda\in(\bigoplus_{i\in I}X_i)^*\), there exist
\(\Lambda_i\in X_i^*\) (\(i\in I\)) such that
\[
\Lambda((x_i)_{i\in I})
=
\sum_{i\in I}\Lambda_i(x_i)
\quad
\big(\forall (x_i)_{i\in I}\in\bigoplus_{i\in I}X_i\big),
\]
\end{lemma}

\begin{proof}
For each \(i\in I\), define
\(\iota_i:X_i\to\bigoplus_{j\in I}X_j\) by
\[
(\iota_i x)_j=
\begin{cases}
x, & j=i,\\
0, & j\ne i.
\end{cases}
\]
Let \(\Lambda\in(\bigoplus_{i\in I}X_i)^*\), and define
\(\Lambda_i:=\Lambda\circ\iota_i\). For any
\(x=(x_i)_{i\in I}\in\bigoplus_{i\in I}X_i\), since
the set \(\mathscr{S}(x)=\{i\in I:x_i\ne0\}\) is finite,
we have
\[
x=\sum_{i\in \mathscr{S}(x)}\iota_i(x_i),
\]
and thus
\[
\Lambda(x)=\sum_{i\in \mathscr{S}(x)}\Lambda(\iota_i(x_i))
=\sum_{i\in I}\Lambda_i(x_i).
\]
\end{proof}

If \(I\) is a countably infinite set, \(X=\prod_{i\in I}E_i\) is endowed
with the product topology, each \(E_i\) is finite-dimensional, and
\(V=\bigoplus_{i\in I}E_i^*\), then
\cite[Chapter~IV, $\S$~4, Theorem~4.3]{SchaeferWolff} gives
\[
\{\phi\in X^*:\phi\text{ is continuous}\}
=
\bigoplus_{i\in I}E_i^*.
\]
Moreover, we have
\begin{itemize}
  \item If \(T\in L(X)\), then for each \(\phi\in V\), the functional
\(T^*\phi=\phi\circ T\) is continuous on \(X\), and thus
\(T^*(V)\subset V\). Hence \(T^*|_V:V\to V\) is a linear operator.
  \item If \(T\in GL(X)\), then also \((T^{-1})^*(V)\subset V\).
  Moreover, for each
\(\phi\in V\),
\[
T^*((T^{-1})^*\phi)=\phi,\quad
(T^{-1})^*(T^*\phi)=\phi.
\]
Hence \(T^*|_V:V\to V\) is a bijective
linear operator, with inverse \((T^{-1})^*|_V\).
\end{itemize}

\begin{theorem}\label{thm:countable-products-no-metric-expansive}
Let \(I\) be a countably infinite set.  For each \(i\in I\), let
\(E_i\) be a nonzero finite-dimensional normed space over
\(\mathbb{K}\).  Define
\[
X=\prod_{i\in I}E_i
\]
with the product topology.  For each compatible metric \(d\) on \(X\),
the following hold:
\begin{enumerate}[{\rm (i)}]
  \item No
operator \(T\in L(X)\) is metrically positively expansive on \((X,d)\).
  \item No
operator \(T\in GL(X)\) is metrically expansive on \((X,d)\).
\end{enumerate}
\end{theorem}

\begin{proof}
Fix a compatible metric \(d\) on \(X\). Let \(V\) be the algebraic direct sum
of $E_i^{*}$, i.e.,
\[
V=\bigoplus_{i\in I}E_i^*.
\]

(i)
Let \(T\in L(X)\). To show that \(T\) is not metrically positively
expansive, it suffices to verify that for any \(\delta>0\), there exists
\(0\ne y\in X\) such that
\[
d(T^n y,0)<\delta
\quad(\forall n\in\mathbb{N}_0).
\]

By compatibility, since \(\{x\in X:d(x,0)<\delta\}\) is an open
neighborhood of \(0\), there exists a nonempty finite set
\(F\subset I\) and
\(\xi>0\) such that
\begin{equation}\label{eq:Choice-xi}
\{x\in X:p_F(x)<\xi\}
\subset
\{x\in X:d(x,0)<\delta\},
\end{equation}
where
\(p_F(x)=\max_{i\in F}\|x_i\|.\)

Let
\[
V_F=\bigoplus_{i\in F}E_i^*\oplus
\bigoplus_{i\in I\setminus F}\{0\}\subset V.
\]
Applying Lemma~\ref{lem:bounded-orbit-functionals} to the vector space \(V\),
the linear operator \(T_0=T^*|_V\), and the finite-dimensional linear subspace
\(V_F\),
there exists a nonzero functional \(\Lambda\in V^*\) such that
\begin{equation}\label{eq:Lamda_Condition}
\sup_{n\in\mathbb{N}_0}
|\Lambda((T^*)^n\phi)|
<+\infty
\quad
(\forall\,\phi\in V_F).
\end{equation}

Since each \(E_i\) is finite-dimensional, the evaluation map
\begin{equation}\label{eq:Def-J_i}
J_i:E_i\to (E_i^*)^*,\quad
J_i(x_i)(\phi_i)=\phi_i(x_i)
\quad(x_i\in E_i,\ \phi_i\in E_i^*)
\end{equation}
is an isomorphism by \cite[Sec.~11.3, Theorem~19]{DF}.
Applying Lemma~\ref{lem:dual-of-algebraic-direct-sum} to
\(\Lambda\in V^*\), with \(X_i=E_i^*\),
it follows that there exists
\(\Lambda_i\in(E_i^*)^*\) (\(i\in I\)) such that, for any
\(\phi=(\phi_i)_{i\in I}\in V\),
\begin{equation}\label{eq:Lamda_phi-Sum}
\Lambda(\phi)=\sum_{i\in I}\Lambda_i(\phi_i).
\end{equation}
Since each \(J_i\) is surjective, for \(\Lambda_i\in(E_i^*)^*\),
there exists $\hat{x}_i\in E_i$ such that $J_i(\hat{x}_i)
=\Lambda_i$. This, together with \eqref{eq:Def-J_i}, implies
\begin{equation}\label{eq:Lamda_phi}
\Lambda_i(\phi_i)=\phi_i(\hat{x}_i)
\quad(\forall \phi_i\in E_i^*,\ \forall i\in I).
\end{equation}
Choose \(\hat{x}=(\hat{x}_i)_{i\in I}\in X\). For any
\(\phi=(\phi_i)_{i\in I}\in V\), applying \eqref{eq:Lamda_phi-Sum}
and \eqref{eq:Lamda_phi} yields
\[
\Lambda(\phi)
=\sum_{i\in I}\Lambda_i(\phi_i)
=\sum_{i\in I}\phi_i(\hat{x}_i)
=\phi(\hat{x}).
\]
Since \(\Lambda\ne0\), we have \(\hat{x}\ne0\).
Together with \eqref{eq:Lamda_Condition}
and \((T^*)^n=(T^n)^*\),
we have
\begin{equation}\label{eq:Boundedness-dual}
\sup_{n\in\mathbb{N}_0}|\phi(T^n \hat{x})|
=\sup_{n\in\mathbb{N}_0}|((T^n)^{*}\phi)(\hat{x})|
=\sup_{n\in\mathbb{N}_0}|((T^*)^n\phi)(\hat{x})|
=\sup_{n\in\mathbb{N}_0}|\Lambda((T^*)^n\phi)|
<+\infty
\quad(\forall \phi\in V_F).
\end{equation}
Fix a basis \(u_1,\ldots,u_{n_1}\) of
the finite-dimensional subspace \(\prod_{i\in F}E_i\times \prod_{i\in I\setminus F}\{0\}\).
For each \(x\in \prod_{i\in F}E_i\times \prod_{i\in I\setminus F}\{0\}\), write uniquely
\[
x=\sum_{j=1}^{n_1}\alpha_j(x)u_j
\quad(\alpha_j(x)\in\mathbb K),
\]
and define $\psi_j:
\prod_{i\in F}E_i\times \prod_{i\in I\setminus F}\{0\}\to \mathbb{K}$
($1\leq j\leq n_1$) by
\[
\psi_j(x)=\alpha_j(x).
\]
Then
\(\psi_j\in
\left(\prod_{i\in F}E_i\times \prod_{i\in I\setminus F}\{0\}\right)^*\).
Since \(x\mapsto \max_{1\le j\le n_1}|\psi_j(x)|\)
and \(p_F\) are norms on this finite-dimensional subspace, the equivalence of
norms on finite-dimensional spaces gives \(C>0\) such that
\begin{equation}\label{eq:Exist-C-psi-bound}
p_F(x)\le C\cdot\max_{1\le j\le n_1}|\psi_j(x)|
\quad\bigg(\forall x\in \prod_{i\in F}E_i\times
\prod_{i\in I\setminus F}\{0\}\bigg).
\end{equation}
Let \(\pi_F:X\to\prod_{i\in F}E_i\times \prod_{i\in I\setminus F}\{0\}\)
be the projection to the subspace
$\prod_{i\in F}E_i\times \prod_{i\in I\setminus F}\{0\}$.
Then $\psi_j\circ \pi_F\in V_F$ for \(1\leq j\leq n_1\).
Since $p_{F}(x)=p_{F}(\pi_F(x))$ for all
$x\in X$, \eqref{eq:Exist-C-psi-bound} and
\eqref{eq:Boundedness-dual} give
\begin{align*}
\sup_{n\in\mathbb{N}_0}p_F(T^n \hat{x})
= & \sup_{n\in\mathbb{N}_0}p_F(\pi_F(T^n \hat{x}))
\le C\cdot \max_{1\le j\le n_1}
\sup_{n\in\mathbb{N}_0}|\psi_j(\pi_F(T^n\hat{x}))|\\
= &~ C\cdot \max_{1\le j\le n_1}
\sup_{n\in\mathbb{N}_0}|\psi_j\circ \pi_F(T^n\hat{x})|
<+\infty.
\end{align*}
Choose
\[
\hat{y}=\frac{\xi\cdot\hat{x}}{1+\sup_{n\in\mathbb{N}_0}p_F(T^n \hat{x})}
\in X.
\]
Then \(\hat{y}\ne0\), and
\[
p_F(T^n \hat{y})=\frac{\xi}{1+\sup_{n\in\mathbb{N}_0}p_F(T^n \hat{x})}\cdot p_F(T^n \hat{x})<\xi
\quad(\forall n\in\mathbb{N}_0).
\]
Together with \eqref{eq:Choice-xi}, we have
\[
d(T^n \hat{y},0)<\delta
\quad(\forall n\in\mathbb{N}_0).
\]
Hence \(T\) is not metrically positively expansive.

\medskip

(ii) By Lemma~\ref{lem:bounded-two-sided-orbit-functionals}, similarly to
the proof of {\rm (i)}, it can be verified that each \(T\in GL(X)\) is not
metrically expansive.
\end{proof}

\begin{remark}\label{rem:finite-dimensional-product-scope}
The proof of Theorem~\ref{thm:countable-products-no-metric-expansive}
uses the finite-dimensionality of the spaces \(E_i\) in the following
finite-coordinate facts.  First, for each nonempty finite set
\(F\subset I\), the subspace
\[
V_F=\bigoplus_{i\in F}E_i^*
\oplus
\bigoplus_{i\in I\setminus F}\{0\}
\]
is finite-dimensional, which is required in
Lemmas~\ref{lem:bounded-orbit-functionals}
and~\ref{lem:bounded-two-sided-orbit-functionals}.
Second, for each \(i\in I\), the evaluation map
\[
J_i:E_i\to(E_i^*)^*,
\quad
J_i(x_i)(\phi_i)=\phi_i(x_i),
\]
is surjective.  This is used, together with
Lemma~\ref{lem:dual-of-algebraic-direct-sum},
to write each \(\Lambda\in V^*\) in the form
\[
\Lambda(\phi)=\sum_{i\in I}\phi_i(\hat{x}_i)
\quad
\big(\phi=(\phi_i)_{i\in I}\in V\big).
\]
Third, finite-dimensionality is used through the equivalence of norms on
\(\prod_{i\in F}E_i\times \prod_{i\in I\setminus F}\{0\}\), which gives
\eqref{eq:Exist-C-psi-bound}.  These are precisely the points at which
the proof does not extend directly to products of infinite-dimensional
spaces.
\end{remark}

\begin{remark}
\label{rem:omega-A-positive-B-negative}
The space \(\mathbb{K}^{\mathbb{Z}}\), endowed with the product topology,
gives different answers to Problems~\ref{problem:A}
and~\ref{problem:B}: Theorem~\ref{thm:omega-weighted-shift-shadowing}
gives linear homeomorphisms with the shadowing property, whereas
Theorem~\ref{thm:countable-products-no-metric-expansive} excludes
metric expansivity for each linear homeomorphism under
any compatible metric.
\end{remark}


\begin{thebibliography}{99}
\small
\setlength{\itemsep}{0pt}
\setlength{\parskip}{0pt}
\setlength{\parsep}{0pt}


\bibitem{AlvesBernardesMessaoudi}
F. F. Alves, N. C. Bernardes Jr. and A. Messaoudi,
\emph{Chain recurrence and average shadowing in dynamics},
Monatsh. Math. 196 (2021), 665--697.

\bibitem{AntunesMantovaniVarao}
M. B. Antunes, G. E. Mantovani and R. Var\~ao,
\emph{Chain recurrence and positive shadowing in linear dynamics},
J. Math. Anal. Appl. 506 (2022), Article 125622.

\bibitem{AokiHiraide}
N. Aoki and K. Hiraide,
\emph{Topological Theory of Dynamical Systems--Recent Advances},
North-Holland Mathematical Library, 52, North-Holland, Amsterdam, 1994.

\bibitem{BayartMatheron}
F. Bayart and E. Matheron,
\emph{Dynamics of Linear Operators},
Cambridge Tracts in Mathematics, 179, Cambridge University Press, Cambridge,
2009.

\bibitem{BermudezBonillaMartinezPeris}
T. Berm\'udez, A. Bonilla, F. Mart\'inez-Gim\'enez and A. Peris,
\emph{Li-Yorke and distributionally chaotic operators},
J. Math. Anal. Appl. 373 (2011), 83--93.

\bibitem{BernardesBonillaMullerPerisDistributional}
N. C. Bernardes Jr., A. Bonilla, V. M\"uller and A. Peris,
\emph{Distributional chaos for linear operators},
J. Funct. Anal. 265 (2013), 2143--2163.

\bibitem{BernardesBonillaMullerPerisLiYorke}
N. C. Bernardes Jr., A. Bonilla, V. M\"uller and A. Peris,
\emph{Li-Yorke chaos in linear dynamics},
Ergodic Theory Dynam. Systems 35 (2015), 1723--1745.

\bibitem{BernardesBonillaPerisMean}
N. C. Bernardes Jr., A. Bonilla and A. Peris,
\emph{Mean Li-Yorke chaos in Banach spaces},
J. Funct. Anal. 278 (2020), Article 108343.

\bibitem{BCDFP}
N. C. Bernardes Jr., B. M. Caraballo, U. B. Darji, V. V. F\'avaro and A. Peris,
\emph{Generalized hyperbolicity, stability and expansivity for operators on
locally convex spaces}, J. Funct. Anal. 288 (2025), 110696.

\bibitem{BCDMP}
N. C. Bernardes Jr., P. R. Cirilo, U. B. Darji, A. Messaoudi and E. R. Pujals,
\emph{Expansivity and shadowing in linear dynamics},
J. Math. Anal. Appl. 461 (2018), 796--816.

\bibitem{BernardesMessaoudiGrobman}
N. C. Bernardes Jr. and A. Messaoudi,
\emph{A generalized Grobman-Hartman theorem},
Proc. Amer. Math. Soc. 148 (2020), 4351--4360.

\bibitem{BernardesMessaoudi}
N. C. Bernardes Jr. and A. Messaoudi,
\emph{Shadowing and structural stability for operators},
Ergodic Theory Dynam. Systems 41 (2021), 961--980.

\bibitem{BernardesPeris}
N. C. Bernardes Jr. and A. Peris,
\emph{On shadowing and chain recurrence in linear dynamics},
Adv. Math. 441 (2024), 109539.

\bibitem{BonetFrerickPerisWengenroth}
J. Bonet, L. Frerick, A. Peris and J. Wengenroth,
\emph{Transitive and hypercyclic operators on locally convex spaces},
Bull. Lond. Math. Soc. 37 (2005), 254--264.

\bibitem{BonetKalmesPeris}
J. Bonet, T. Kalmes and A. Peris,
\emph{Dynamics of shift operators on non-metrizable sequence spaces},
Rev. Mat. Iberoam. 37 (2021), 2373--2397.

\bibitem{BonetPeris}
J. Bonet and A. Peris,
\emph{Hypercyclic operators on non-normable Fr\'echet spaces},
J. Funct. Anal. 159 (1998), 587--595.

\bibitem{CiriloGollobitPujals}
P. R. Cirilo, B. Gollobit and E. R. Pujals,
\emph{Dynamics of generalized hyperbolic linear operators},
Adv. Math. 387 (2021), Article 107830.

\bibitem{DF}
D. S. Dummit and R. M. Foote,
\emph{Abstract Algebra},
3rd edition,
John Wiley \& Sons, Hoboken, NJ, 2004.

\bibitem{Eisenberg}
M. Eisenberg,
\emph{Expansive automorphisms of finite-dimensional vector spaces},
Fund. Math. 59 (1966), 307--312.

\bibitem{EisenbergHedlund}
M. Eisenberg and J. H. Hedlund,
\emph{Expansive automorphisms of Banach spaces},
Pacific J. Math. 34 (1970), 647--656.

\bibitem{GrosseErdmannPeris}
K.-G. Grosse-Erdmann and A. Peris,
\emph{Linear Chaos},
Universitext, Springer, London, 2011.

\bibitem{GrosseErdmannPerisWeaklyMixing}
K.-G. Grosse-Erdmann and A. Peris,
\emph{Weakly mixing operators on topological vector spaces},
Rev. R. Acad. Cienc. Exactas Fis. Nat. Ser. A Mat. RACSAM 104 (2010),
413--426.

\bibitem{Hedlund}
J. H. Hedlund,
\emph{Expansive automorphisms of Banach spaces. II},
Pacific J. Math. 36 (1971), 671--675.

\bibitem{Mazur}
M. Mazur,
\emph{Hyperbolicity, expansivity and shadowing for the class of normal
operators},
Funct. Differ. Equ. 7 (2000), 147--156.

\bibitem{Ombach}
J. Ombach,
\emph{The shadowing lemma in the linear case},
Univ. Iagel. Acta Math. 31 (1994), 69--74.

\bibitem{PerisHypercyclicityCriterion}
A. Peris,
\emph{A hypercyclicity criterion for non-metrizable topological vector spaces},
Funct. Approx. Comment. Math. 59 (2018), 279--284.

\bibitem{Pilyugin}
S. Yu. Pilyugin,
\emph{Shadowing in Dynamical Systems},
Lecture Notes in Mathematics, 1706, Springer, Berlin, 1999.

\bibitem{SchaeferWolff}
H. H. Schaefer and M. P. Wolff,
\emph{Topological Vector Spaces},
2nd edition, Graduate Texts in Mathematics, Vol.~3,
Springer, New York, 1999.

\bibitem{Walters}
P. Walters,
\emph{An Introduction to Ergodic Theory},
Graduate Texts in Mathematics, 79, Springer, New York, 1982.

\end{thebibliography}
\end{document}